\documentclass[11pt]{amsart}

\usepackage{mystyle}

\title{Regularity for fully non linear equations with non local drift}
\author[H. A. Chang Lara]{H\'ector A. Chang Lara}

\begin{document}

\begin{abstract}
We show H\"older regularity of solutions of elliptic integro-differential equations appearing in stochastic optimal control. The operators are assumed elliptic with respect to a family of linear operators obtained as a convolution with kernels which are not non necessarily symmetric. The non local drift is provided by the odd part of the kernel which is assumed to have an order of scaling smaller than or equal to the even part and larger than or equal to one. In particular we are able to handle the equation $\D^{1/2}u + |Du| = 0$ were the diffusion and drift term have the same order. 
\end{abstract}

\maketitle


\section{Introduction}\label{Sec:Introduction}

In this work we study the regularity of viscosity solutions of non divergence integro-differential equations with no symmetry assumptions to be explained later on this introduction.

A fairly general model to have in mind are equations given by operators of the form,
\begin{align*}
Iu(x) = \inf_{\b\in B}\sup_{\a\in A}L_{K_{\a,\b}}u(x),
\end{align*}
where $\{L_{K_{\a,\b}}\}_{\a\in A,\b\in B}$ is a family of linear integro-differential operators computed by,
\begin{align}\label{eq:lineal}
L_{K_{\a,\b}}u &= \int \d(u,x;y)K_{\a,\b}(y)dy,\\
\d(u,x;y) &= u(x+y) - u(x) - Du(x)\cdot y\chi_{B_1}(y),
\end{align}
with some uniform hypothesis over the kernels $\{K_{\a,\b}\}_{\a\in A,\b\in B}$. The Dirichlet boundary problem,
\begin{align*}
Iu &= 0 \text{ in $\W\ss\R^n$},\\
u &= g \text{ in $\R^n\sm\W$},
\end{align*}
arise in optimal control models driven by purely jump stochastic processes where $K_{\a,\b}(y)dy$ is a family of Levy measures. The description of this and many other types of control problems can be found in the book by M. Soner and W. Fleming, \cite{Soner06}. 

Our techniques follow very closely the classical proof of the Krylov-Safonov regularity for fully non linear second order differential equations, see \cite{Krylov79} and the book \cite{Caffarelli95}. They also recover the classical H\"older regularity when the order of the equation goes to two. We describe next a brief historical account about this type of problems in which our results are also based.

Integro-differential equations have been studied since long time ago with probabilistic and analytic techniques. Many models have been posed that involve integro-differential equations, we already mentioned a broad family of them coming from optimal control. Recently many works have been published where the authors approach the regularity using analytic techniques, our results belong to this class. The paper by G. Barles and C. Imbert \cite{Barles08} revisit the viscosity theory of integro-differential equations and establish a fairly general comparison principle for them. The work of L. Silvestre in \cite{Silvestre06} shows us how to obtain a point estimate at every scale that allows us to get a diminish of oscillation and then the H\"older regularity of the solution. Even thought the proof in \cite{Silvestre06} is simple an elegant it was not powerful enough to recover the second order theory as the order of the equation went to two. In the series of papers \cite{Caffarelli09,Caffarelli10,Caffarelli11} L. Caffarelli and L. Silvestre developed the equivalent Krylov-Safanov Harnack inequality, Cordes-Nirenberg estimates and the Evans-Krylov estimates which allowed them to show that viscosity solutions were classical in a board range of equations, involving concave (or convex) operators. Moreover their estimates remain uniform as the order of the equation goes to two showing that the regularity theory of second order equation can be extended to integro-differential problems.

In all the previous works by L. Caffarelli and L. Silvestre there is always a symmetry assumption in the linear operators. In \cite{Silvestre06} it is mentioned that this assumption was made only to make the exposition cleaner but in \cite{Caffarelli09,Caffarelli10,Caffarelli11} it becomes more relevant as the scalings of such operators may begin to bring additional terms that also have to be controlled. Specifically, if we rescale $u$ by $\tilde u = u(rx)$ and $u$ satisfies $L_Ku = f$, with $L_K$ coming from \eqref{eq:lineal}, then the change of variables formula says that $\tilde u$ satisfies,
\begin{align*}
\int \1\tilde u(x+y) - \tilde u(x) - D\tilde u(x)\cdot y\chi_{B_{r^{-1}}}(y)\2r^nK(ry)dy = f(rx).
\end{align*}
Notice that the characteristic function has changed its support from $B_1$ to $B_{r^{-1}}$. For $r<1$ this brings an additional gradient term if we split $B_{r^{-1}}$ as the disjoint union of $B_1$ and $B_{r^{-1}}\sm B_1$.

A possible way to deal with this problematic term is by assuming that $K$ is even which is the case in \cite{Silvestre06,Caffarelli09,Caffarelli10,Caffarelli11}. Another possibility is by assuming that the odd part of $K$ has enough integrability such that the operator can be splited as,
\begin{align*}
L_Ku &= L_{K_e}u + L_{K_o}u,\\
L_{K_e}u &= \frac{1}{2}\int (u(x+y)+u(x-y)-2u(x))K_e(y)dy,\\
L_{K_o}u &= \frac{1}{2}\int (u(x+y)-u(x-y))K_o(y)dy,
\end{align*}
where $K_e$ and $K_o$ are the even and odd parts of $K$. This is the case of \cite{Davila12} where it is also assumed that the order of the odd part is strictly smaller than the order of the even part. In such work the argument to obtain the H\"older regularity is based on considering the non local drift as a perturbation term to the equation.

In this work we do not assume that such decomposition is possible. Instead we see how to control the new gradient term by considering a larger class of linear operators which already included a gradient term. The regularity in this case is expected because at small scales the total kernel always remain positive which implies in particular that the even part controls the odd part. Our main contribution relies on a modification of the Aleksandrov-Bakelman-Pucci (ABP) estimate from \cite{Caffarelli09}. We consider a barrier (Lemma \ref{lemma:Barrier2}) that allows us to localize the estimate, disregarding the influence of the gradient term, assuming that the right hand side was already localized in a small ball (this is the role of the special function provided in Lemma \ref{lemma:Barrier3}). As a consequence we obtain a point estimate stated in Theorem \ref{thm:Point_estimate} and the H\"older regularity estimates stated in Theorems \ref{thm:holder} and \ref{thm:holder_gradient}.

On the preliminary Section \ref{Sec:Preliminaries} we discuss the definitions of ellipticity and viscosity solutions. We take some time in explaining how to set our hypothesis in order to have enough control at small scales. Also in that section we discuss the comparison principle and the existence of solutions of the Dirichlet problem by Perron's method. Section \ref{Sec:ABP_estimate} covers the ABP estimate which follows the same ideas as in \cite{Caffarelli09} but have also to be adapted in order to consider the gradient term. After having an ABP estimate it is fairly well known how to get Harnack estimates and H\"older regularity, in Section \ref{Sec:Holder_regularity} we go directly to prove a point estimate (or $L^\epsilon$ Lemma, or weak Harnack) and then we just state the regularity Theorems we are allow to get from there.


\subsection{Notation}\label{Subsec:Notation}

\begin{align*}
B_r(x) &= \{y\in\R^n:|y-x|<r\},\\
B_r &= B_r(0),\\
Q_d(x) &= (x_1-d/(2\sqrt{n}),x_1+d/(2\sqrt{n}))\times\ldots\times(x_n-d/(2\sqrt{n}),x_n+d/(2\sqrt{n})),\\
aQ_d(x) &= Q_{ad}(x). 
\end{align*}

For a set $A \ss \R^n$ we denote by $\chi_A$ is characteristic function and $\diam(A)$ its diameter.

For a function $u$ we denote:

\begin{enumerate} 
\item $Du$ its gradient.
\item $u^+ = \max(u,0)$ and $u^- = -\min(u,0)$.
\end{enumerate}

The set of upper (lower) semicontinuous functions in $\W \ss \R^n$ will be denoted by $USC(\W)$ ($LSC(\W)$) respectively.

For a measure $\w$ defined in $\R^n$, the set of functions integrable with respect to $\w$ will be denoted by $L^1(\w)$. In particular we will use $\w_\s(dy) = \min(|y|^{-(n+\s)},1)dy$.

Given a family of linear operators $\cL$ we denote by $\cM^\pm$ the extremal Pucci operators computed at a given function $u$ at the point $x$ by,
\begin{align*}
\cM^+_\cL u(x) = \sup_{L\in\cL}Lu(x),\\
\cM^-_\cL u(x) = \inf_{L\in\cL}Lu(x).
\end{align*}

The letter $C$ will usually denote universal constants that may vary from line to line.


\section{Preliminaries}\label{Sec:Preliminaries}

We start this Section by motivating the hypothesis we will impose on our operator $I$. We start with the ellipticity which will allow to set those conditions not directly on $I$ but on a family of linear operators $\cL$ controlling $I$. The particular observations that we need to keep in mind are a positivity condition which would imply that there is enough diffusion and and scaling condition which will allow to get H\"older regularity from a diminish of oscillation result obtained at scale one.

After discussing the hypothesis on $\cL$ we introduce the viscosity solutions associated with integro-differential equations. These type of solutions are constructed such that they satisfy the maximum principle. Fundamental properties as the stability, existence and uniqueness are discussed in this section too and in most of the cases refer the proof to the already well know works in \cite{Barles08,Barles08-2,Caffarelli09}.

In a future project we also plan to answer some of the basic questions related with viscosity solutions which remain open up to this day. To have such a reference available would have made this section substantially shorter. At the end of this section we discuss what we have in mind related to this issue.

\subsection{Elliptic operators}\label{Subsec:Elliptic_operators}

The first concept we introduce is the one of ellipticity with respect to a family of linear operators. The definition given here is essential the same one as the one in \cite{Caffarelli09}. However the ellipticity notion classically refers to some positivity condition here we use it as a way to control the non linearity by linear operators with no sign condition. Later, when we define the particular family of linear operators we will see a positivity condition which finally justifies the use of the name.

Extremal operators $\cM^-_\cL$ and $\cM^-_\cL$ with respect to a family of linear operators $\cL$, defined over a domain $\W\ss\R^n$, are constructed by,
\begin{align*}
\cM^+_\cL u(x) = \sup_{L\in\cL}Lu(x),\\
\cM^-_\cL u(x) = \inf_{L\in\cL}Lu(x).
\end{align*}

\begin{definition}[Ellipticity]\label{def:Ellipticity}
An operator $I$, defined over a domain $\W\ss\R^n$, is elliptic with respect to a family of linear operators $\cL$ if for every $x\in\W$ and any pair of functions $u$ and $v$ where $Iu(x)$ and $Iv(x)$ can be evaluated then also $Lu(x)$ and $Lv(x)$ are well defined and
\begin{align*}
\cM^-_\cL(u-v)(x) \leq Iu(x) - Iv(x) \leq \cM^+_\cL(u-v)(x).
\end{align*}
\end{definition}

To be concrete we need to fix what are going to be the linear operators we consider. Initially we are interested in non local operators $L_K$, defined in terms of a kernel $K:\R^n\to\R$ in the following way
\begin{align}\label{non_local_linear}
L_K u = \int \1u(x+y) - u(x) - Du(x)\cdot y\chi_{B_1} \2K(y)dy.
\end{align}
For $u$ sufficiently regular at $x$ and bounded this makes sense if,
\begin{align}\label{integrability_condition}
\int \min(1,|y|^2)|K(y)| dy < \8.
\end{align}


\subsubsection{Diffusion versus drift}

The (preservation of the) sign of $K$ around the origin plays an important role. For instance, if $K$ is positive (or negative) $L_Ku(x)$ measures some sort of deviation of $u$ from an average of itself. Specifically,
\begin{align*}
L_Ku(x) = \lim_{\e\to0}\int_{\R^n\sm B_\e} (u(x+y) - u(x))K(y)dy.
\end{align*}

Knowing that $L_Ku = 0$ is a singular version of the mean value property which allowed to show H\"older regularity in \cite{Silvestre06}. Still the same techniques as in \cite{Silvestre06} hold if $K$ is only assumed positive around the origin as the influence of the tail of $K$ can always be pass as a right hand side and dilations looking at smaller scales will spread the positivity of the kernel around the origin, see also Section 14 in \cite{Caffarelli09}.

If $K$ takes positive and negative values around the origin then we can not always expect the solvability of the Dirichlet problem. An example of this phenomena can be taken from a slight modifications of the counterexample in Section 5 in \cite{Barles08-2}.

The previous observations can be also made in terms of the even/odd decomposition of $K = K_e + K_o$ respectively. If $K_e$ preserves a sign, $L_{K_e}u(x)$ measures a deviation of $u(x)$ from an average of itself centered at $x$, this is a diffusive term. On the other hand, $L_{K_o}u(x)$ measures an average of the slopes $u$ centered at $x$, this is a drift term.


\subsubsection{Scaling}

An important ingredient of the regularity theory is scale invariance. A diminish of oscillation estimate works to prove regularity of the solution because they also hold also at smaller scales. We say that the equation $L u = f$ has scale $\s$ if the same type of equation gets preserved by a rescaling of $u$ of the form $\tilde u(x) = r^{-\a}u(rx)$ with $r\in(0,1)$ and $\a \in [0,\s]$. This is the case of the Poisson's equation with $\s=2$, therefore the same estimates one obtains for $u$ at scale one can be applied, with the same or better constants, to any scaling of the form $\tilde u(x) = r^{-\a}u(rx)$, with $r \in (0,1)$ and $\a \in (0,2]$.

By the change of variables formula we can also give an explicit form of how $L_K$ gets rescaled. Consider $L_K u = f$ in $\W$, then the rescaled function $\tilde u(x) = r^{-\a} u(rx)$ satisfies,
\begin{align}\label{scaling}
L_{r^{n+\a}K(r\cdot)} \tilde u + \1r^{\a-1}\int_{B_1\sm B_r} y K(y)dy\2\cdot D\tilde u = f(r\cdot) \text{ in $r^{-1}\W$}.
\end{align}

It comes immediately to our attention the gradient term which depends actually on the odd part of $K$. The particular interest of our work comes when this term persists even at smaller scales. This suggest that from the beginning we should had considered $Lu = L_Ku + b\cdot Du$. Moreover, if we expect to prove that there are classical solutions then we should expect the equation to have order at least one.

The second observation is how the kernel gets rescaled, $\tilde K = r^{n+\a}K(r\cdot)$. For $r$ small this might break any uniform positivity assumption around the origin unless $K$ have a singularity of order at least $-(n+\a)$. On one hand, the integrability condition \eqref{integrability_condition} already imposes an upper bound to the growth of $K$ at the origin so that the previous $\a$ have to be less than two. On the other hand, and as we said before, the order has to be at least one if we expect classical solutions, so that by plugging $\a=1$ above we notice that the singularity of $K$ around the origin has to be of order at least $-(n+1)$, in order to the diffusive term to compete against the drift.

In the case when $\a=\s$, the diffusive term in \eqref{scaling}, which is contained in $L_{r^{n+\s}K(r\cdot)} \tilde u$, might be just bounded as $r\to0$. By the same reasoning as in the previous paragraph we also should impose an upper bound on the drift term in order to not degenerate the equation. This is given if for some $\b>0$,
\begin{align*}
r^{\s-1}\left|\int_{B_1\sm B_r} y K(x,y)dy\right| < \b \text{ uniformly in $r$}.
\end{align*}

A particular scaling invariant operator of order $\s\in(0,2)$ is given by the fractional powers of the laplacian. They can be computed by the following expression (modulus a positive universal constant),
\begin{align*}
\D^\s u(x) = (2-\s)\lim_{\e\to0}\int_{\R^n\sm B_e} \frac{u(x+y)-u(x)}{|y|^{n+\s}}dy.
\end{align*}
The factor $(2-\s)$ becomes relevant in order to extend the definition to $\s=2$. When $\s\to2^-$ the integral becomes divergent however the factor $(2-\s)$ tames the singular behavior recovering the actual laplacian of $u$ modulus a positive universal constant.

In the hypothesis we will introduce in the next part we will see that we are actually bounding our kernels by multiples of the kernel of the fractional laplacian.


\subsubsection{Hypothesis}\label{Subsubsec:Hypothesis}

We resume our hypothesis on the family $\cL$ depending on a family of kernels $\cK$ and some additional parameters $\s,\l,\L$ and $\b$, in the following way:

\begin{enumerate}
\item Every $L \in \cL$ is of the form $L = L_K + b\cdot D$ for $K \in \cK$
\item\label{hypothesis_even} For some fixed $1\leq\s<2$, $\l\leq\L$ and the extremal kernels,
\begin{align*}
K^-(y) &= \l(2-\s)|y|^{-(n+\s)}\chi_{B_1}(y),\\
K^+(y) &= \L(2-\s)|y|^{-(n+\s)},
\end{align*}
we have that for every $K \in \cK$, $K^- \leq K \leq K^+$.
\item\label{hypothesis_odd} For some $\b>0$,
\begin{align*}
r^{\s-1}\left|b + \int_{B_1\sm B_r} y K(y)dy\right| \leq \b \text{ for every $r\in(0,1)$}.
\end{align*}
\end{enumerate}

In this case we can also write,
\begin{align*}
\cM^\pm_\cL u(x) = \cM^\pm_\cK u(x) \pm \b|Du(x)|,
\end{align*}
where,
\begin{align*}
\cM^+_\cK u(x) &= \sup_{K\in\cK}L_Ku(x) = (2-\s) \int \frac{\d^+(u,x_0;y)\L - \d^-(u,x_0;y)\l\chi_{B_1}(y)}{|y|^{n+\s}}dy,\\
\cM^-_\cK u(x) &= \inf_{K\in\cK}L_Ku(x) = (2-\s) \int \frac{\d^+(u,x_0;y)\l\chi_{B_1}(y) - \d^-(u,x_0;y)\L}{|y|^{n+\s}}dy.
\end{align*}
We fix this notation for future references.
 
Notice that (\ref{hypothesis_odd}) gets reduced to $b$ being uniformly bounded in $\cL$ if $\s>1$. An interesting case arises when $\s=1$. A particular family of operator that satisfies all the previous hypothesis are $\cL = \{L_b = \D^{1/2} + b\cdot D: |b|\leq\b\}$. This is one of the simplest cases where our regularity results apply to equations where the drift and the diffusion scale with the same order.

The upper bound $K^+$ is necessary in order to construct barriers and control the solution. The factor $(2-\s)$ allows us to build uniform estimates as $\s$ goes to 2 which is the case of second order equations. We include a Lemma at the end of this part that discusses this aspect.

Finally notice that the family $\cL$ is scaling invariant of order $\s$. Given $\a\in(0,\s]$, $r\in(0,1)$, $L = L_K + b\cdot D \in \cL$ and $u$ such that $Lu = f$ in $\W$ then the rescaling $\tilde u(x) = r^{-\a}u(rx)$ satisfies $\tilde L \tilde u = \tilde f$ in $r^{-1}\W$ where,
\begin{align*}
\tilde f &= r^{\s-\a}f(r\cdot) \text{ (notice that $r^{\s-\a} \leq 1$)},\\
\tilde L &= L_{\tilde K} + \tilde b\cdot D,\\
\tilde K &= r^{n+\s}K(r\cdot),\\
\tilde b &= r^{\s-1}\1 b + \int_{B_1\sm B_r} y K(y)dy\2.
\end{align*}
For $\r \in (0,1)$,
\begin{align*}
&\r^{\s-1}\left|\tilde b  + \r^{\s-1}\int_{B_1\sm B_\r}y\tilde K(y)dy\right|,\\
&= (r\r)^{\s-1}\left|b + \int_{B_1\sm B_r} y K(y)dy + \int_{B_r\sm B_{r\r}}yK(y)dy\right|,\\
&\leq \b
\end{align*}
Then $\tilde L \in \cL$.


As the order $\s$ goes to two, the extremal operators go to a different type of extremal operators that are still comparable with the classical Pucci extremal operators.

\begin{lemma}[Limit as $\s$ goes to 2]\label{lemma:limit_to_the_classic}
Let $u \in C^2(x_0) \cap L^\8$ then,
\begin{align*}
\lim_{\s\to2} \cM^-_\cK u (x_0) &= |\p B_1|\tilde \cM^- D^2u(x_0),\\
&= \int_{\p B_1} \1\1\theta^t D^2 u(x_0)\theta\2^+ \l - \1\theta^t D^2 u(x_0)\theta\2^- \L \2d\theta.
\end{align*}
In particular, for the classical Pucci operators $\cM^-$ defined as in \cite{Caffarelli95},
\begin{align*}
\tilde \cM^- D^2u(x_0) \leq \cM^- D^2u(x_0) \leq C\tilde \cM^- D^2u(x_0).
\end{align*}
for some universal $C$ depending only on $n$.
\end{lemma}

\begin{remark}
A similar result also holds for $\tilde \cM^+ D^2u(x_0) = -\tilde \cM^- (-D^2u(x_0))$, namely
\begin{align*}
\lim_{\s\to2} \cM^+_\cK u (x_0) &= |\p B_1|\tilde \cM^+ D^2u(x_0),\\
&= \int_{\p B_1} \1\1\theta^t D^2 u(x_0)\theta\2^+ \L - \1\theta^t D^2 u(x_0)\theta\2^- \l \2d\theta.
\end{align*}
and
\begin{align*}
C\tilde \cM^+ D^2u(x_0) \leq \cM^+ D^2u(x_0) \leq \tilde \cM^+ D^2u(x_0).
\end{align*}
\end{remark}

\begin{proof}
We split the integral as,
\begin{align*}
\cM^-_\cK u (x_0) &= (2-\s) \int \frac{\d^+(u,x_0;y)\l\chi_{B_1}(y) - \d^-(u,x_0;y)\L}{|y|^{n+\s}}dy,\\
&= (2-\s) \int_{B_1} + (2-\s)\int_{\R^n \sm B_1}.
\end{align*}
The last term goes to zero as $\s$ goes to 2 because the integral is uniformly bounded from the assumption that $u \in L^\8$. We use now that for $y$ small, $\d^\pm(u,x_0;y) = (y^tD^2u(x_0)y)^\pm + o(|y|^2)$ in the first term. As $\s$ goes to 2 the lower order term vanishes so that we only have to consider,
\begin{align*}
&(2-\s)\int_{B_1} \frac{(y^tD^2u(x_0)y)^+\l - (y^tD^2u(x_0)y)^-\L}{|y|^{n+\s}}dy,\\
&= (2-\s) \int_0^1 r^{n-1} dr \int_{\p B_1} \1(\theta^tD^2u(x_0)\theta)^+\l - (\theta^tD^2u(x_0)\theta)^-\L\2 r^{-(n+\s)+2}d\theta,\\
&= |\p B_1|\tilde\cM^- D^2u(x_0).
\end{align*}

For the last part of the Lemma we check first that $\tilde \cM^- D^2u(x_0) \leq \cM^- D^2u(x_0)$. Notice that for any positive definitive symmetric matrix $M$ such that $\l Id \leq M \leq \L Id$,
\begin{align*}
\trace(MD^2u(x_0)) &= \frac{1}{|\p B_1|}\int_{\p B_1} \theta^t MD^2u(x_0) \theta d\theta,\\
&\geq \frac{1}{|\p B_1|} \int_{\p B_1} \1(\theta^tD^2u(x_0)\theta)^+\l - (\theta^tD^2u(x_0)\theta)^-\L\2 d\theta,\\
&= \tilde\cM^- D^2u(x_0).
\end{align*}
We conclude then by taking the infimum on the left hand side above.

For the other inequality denote $M = D^2u(x_0)$ and decompose $\R^n$ in an appropriated system of coordinates $\R^n = \R^{n_1}\times\R^{n_2}$ such that $M$ is positive over $\R^{n_1}$ and negative over $\R^{n_2}$. If $n_1$ or $n_2$ is zero then we get that $\tilde\cM^\pm M = \cM^\pm M$, therefore we can assume that $n_1$ and $n_2$ are at least one.

Consider then the change of variables $\zeta:\p B_1^{n_1}\times \p B_1^{n_2}\times [0,\pi/2]\to \p B_1^n$ (where $\p B_1^m$ is the unit sphere in $\R^m$) given by,
\begin{align*}
\zeta(\theta_1,\theta_2,\theta) = (\theta_1\cos\theta, \theta_2\sin\theta).
\end{align*}
Its determinant depends only on $\theta$ and is given by
\begin{align*}
\phi_{n_1,n_2}(\theta) = \cos^{n_1-1}\theta\sin^{n_1-1}\theta\sqrt{1+\cos^2\theta\sin^2\theta}.
\end{align*}
Then,
\begin{align*}
\tilde\cM^- M &= \frac{1}{|\p B_1^n|}\int_{\p B_1^n} \1\1\theta^tM\theta\2^+ \l - \1\theta^t M\theta\2^- \L \2d\theta,\\
&= \frac{1}{|\p B_1^n|}\int_0^{\pi/2} \left(\l\cos^2\theta\int_{\p B_1^{n_1}}\theta_1^tM\theta_1d\theta_1\right.,\\
&\left.{} + \L\sin^2\theta\int_{\p B_1^{n_2}}\theta_2^tM\theta_2d\theta_2\right)\phi_{n_1,n_2}(\theta)d\theta,\\
&= \frac{|\p B_1^{n_1}||\p B_1^{n_2}|}{|\p B_1^n|}\int_0^{\pi/2} \1\l \trace M^- \cos^2\theta - \L \trace M^+ \sin^2\theta \2\phi_{n_1,n_2}(\theta)d\theta,\\
&= C_{n_1,n_2}(\l \trace M^- - \L \trace M^+).
\end{align*}
With $C_{n_1,n_2}$ positive and bounded. We conclude by considering all possible partitions of $n = n_1+n_2$ with $n_1,n_2 \in \N$, each one of them at least one.
\end{proof}


\subsection{Formal definition of our operators}\label{Subsec:Formal_definition_of_our_operators}

For every $L \in \cL$ (with the hypothesis from the previous part) $Lu(x)$ can be defined for $u \in C^{1,1}(x) \cap L^1(\w_\s)$ where $\w_\s(dy) = \min(|y|^{-(n+\s)},1)dy$. Moreover $Lu$ is continuous in $B_r(x_0)$ if $u \in C^2(B_r(x_0)) \cap L^1(\w_\s)$. Going back to the non linearity $I$, recall that if $Iu(x)$ is well defined then $Lu(x)$ needs to be well defined too for every $L \in \cL$, given that $I$ is elliptic with respect to $\cL$. It is then reasonable to ask at least $C^{1,1}(x)$ regularity and $L^1(\w_\s)$ integrability in order to evaluate $Iu(x)$. Stability properties of $I$ depend on $Iu$ being continuous when $u$ is sufficiently regular, in this case $C^2(B_r(x_0)) \cap L^1(\w_\s)$ seems to be a reasonable minimum requirement. The following definition comes from \cite{Caffarelli09}.

\begin{definition}[Elliptic continuous operators]\label{def:operators}
$I$ is a continuous operator, elliptic with respect to $\cL = \cL(\cK)$ in $\W$ if,
\begin{enumerate}
 \item $I$ is an elliptic operator with respect to $\cL$ in $\W$,
 \item $Iu(x)$ is well defined for any $u \in C^{1,1}(x) \cap L^1(\w_\s)$ and $x\in\W$,
 \item $Iu$ is continuous in $B_r(x_0)$ for any $u \in C^2(B_r(x_0)) \cap L^1(\w_\s)$ and $B_r(x_0)\ss\W$.
\end{enumerate}
\end{definition}

Translation invariant operators should be such that $Iu$ evaluated at $x$ computes the same number as $I$ evaluated at the translated function $u_z = u(\cdot-z)$ at the point $(x+z)$.


\subsubsection{Examples}

The extremal operators $\cM^\pm_\cL$ satisfy the first two requirements in the Definition \ref{def:operators} and are also invariant by translations. More generally, the same can be said for any inf-sup (or sup-inf) combination of operators in $\cL$, i.e.
\begin{align}\label{inf_sub_operators}
I = \inf_{\a\in A}\sup_{\b\in B} L_{K_{\a,\b}} + b_{\a,\b}\cdot D
\end{align}
An adaptation of the proof of Lemma 4.2 in \cite{Caffarelli09} shows also that all the previous examples 
give also continuous operators.

We also see from here that translation invariant operators $I$ which are elliptic with respect to $\cL$, such that $Iu(x)$ is well defined for any $u \in C^{1,1}(x) \cap L^1(\w_\t)$ is automatically continuous. Indeed, the ellipticity assumption can be evaluated for $u \in C^2(B_r(x_0)) \cap L^1(\w_\t)$ and the translation $u_z = u(z+\cdot) \in C^2(B_r(x_0-z)) \cap L^1(\w_\t)$,
\begin{align*}
Iu(x+z) - Iu(x) &= (Iu_z - Iu)(x),\\
&\in [\cM^-_\cL(u_z-u)(x), \cM^+_\cL(u_z-u)(x)].
\end{align*}
which goes to zero as $z$ approaches $x$. This can be justified by the explicit computations of $\cM^\pm_\cL$.


\subsection{Viscosity solutions}

\begin{definition}[Test functions]\label{def:Test_functions}
Given $\s\in[1,2)$, a test function at $x_0 \in \R^n$ is defined as a pair $(\varphi, B_r(x_0))$, for $r>0$, such that $\varphi \in C^2(B_r(x_0)) \cap L^1(\w_\t)$.
\end{definition}

Usually we omit the domain $B_r(x_0)$ and denote the test function just by $\varphi$.

A test function $u_{(\varphi, B_r(x_0))} \in C^2(B_r(x_0)) \cap L^1(\w_\s)$ can always be constructed from $\varphi \in C^2(B_r(x_0))$ and $u \in L^1(\w_\t)$ in the following way,
\begin{align*}
u_{(\varphi, B_r(x_0))} = \begin{cases}
\varphi &\text{ in $B_r(x_0)$},\\
u &\text{ in $\R^n\sm B_r(x_0)$}.
\end{cases}
\end{align*}
We fix this definition for future references.

\begin{definition}[Viscosity solutions]\label{def:Viscosity_solutions}
Given a non local operator $I$ and a function $f:\W\to\R$ we say that $u \in LSC(\W) \cap L^1(\w_\s)$ is a supersolution (subsolution) to
\begin{align*}
Iu \leq (\geq) f \text{ in the viscosity sense in $\W$},
\end{align*}
if for every test function $(\varphi, B_r(x_0))$ that touches $u$ from below (above), i.e.
\begin{enumerate}
\item $x_0 \in \W$,
\item $\varphi(x_0) = u(x_0)$,
\item $\varphi \leq (\geq) u$ in $\R^n$,
\end{enumerate}
then $I\varphi(x_0) \leq (\geq) f(x_0)$.

Additionally, $u \in C(\W) \cap L^1(\w_\t)$ is a viscosity solution to $Iu = f$ in $\W$ if it is simultaneously a subsolution and a supersolution.
\end{definition}

Notice that for $u \in LSC(x_0) \cap L^1(\w_\s)$ and any $\varphi \in C^2(B_r(x_0))$ such that,
\begin{enumerate}
\item $x_0 \in \W$,
\item $\varphi(x_0) = u(x_0)$,
\item $\varphi \leq u$ in $B_r(x_0)$,
\end{enumerate}
we can always construct a test function touching $u$ from below by considering $u_{(\varphi, B_r(x_0))} \in C^2(B_r(x_0)) \cap L^1(\w_\s)$. In this sense we also call $\varphi \in C^2(B_r(x_0))$ a test function even thought we actually plug $u_{(\varphi, B_r(x_0))}$ into the equation.


\subsection{Qualitative properties}\label{Subsec:Qualitative_properties}

In this section we discuss some of the fundamental properties of the viscosity solution theory. Lemma \ref{lemma:Principal_value_lemma} says that for viscosity solutions of equations involving operators of inf-sup type \eqref{inf_sub_operators}, having contact with a test functions from one side, forces enough regularity to evaluate the same operator in a principal value sense. We get as a corollary that regular enough viscosity solutions are also classical. The results after that are the stability and the comparison principle necessary in order to establish existence and uniqueness of the Dirichlet problem by Perron's method.


\subsubsection{A principal value lemma}\label{Subsubsec:A_principal_value_lemma}

It was already noticed in \cite{Caffarelli09} that for viscosity solutions $u$ of operators $I$, obtained as an inf-sup combination of non local, linear operators, one gets that having contact by a test function from one side forces some regularity on the other side too. See also Section 1.3 in \cite{Barles08}. This turns out to be a very useful tool as we can (almost) evaluate $I$ on $u$ having to recur to the test functions in a minimal way.

\begin{lemma}\label{lemma:Principal_value_lemma}
Let $I$ of the form \eqref{inf_sub_operators} such that $\{L_{\a,\b}\}_{\a\in A, \b\in B} \ss \cL$. Given a viscosity super solution $u \in LSC(\W) \cap L^1(\w_\s)$ of $Iu \leq f$ in $\W$ and $\varphi \in C^2(B_r(x_0))$ such that touches $u$ from below at $x_0 \in \W$, then for each $K\in\cK$ the following limit is well defined,
\begin{align*}
L_K(u,D\varphi(x_0))(x_0) := \lim_{\e\to0} \int_{\R^n \sm B_\e} \d(u,D\varphi(x_0),x_0;y)K(y)dy,
\end{align*}
where
\begin{align*}
\d(u,p,x_0;y) = u(x_0+y) - u(x_0) - p\cdot y\chi_{B_1}(y).
\end{align*}
Moreover,
\begin{align*}
\inf_{\a\in A} \sup_{\b\in B} L_{\a,\b}(u,D\varphi(x_0))(x_0) \leq f(x_0),  
\end{align*}
where,
\begin{align*}
L_{\a,\b}(u,D\varphi(x_0))(x_0) = L_{K_{\a,\b}}(u,D\varphi(x_0))(x_0) + b_{\a,\b}\cdot D\varphi(x_0).  
\end{align*}
\end{lemma}

\begin{remark}
In the particular case that $D\varphi(x_0) = 0$ we can also write,
\begin{align*}
L_K(u,0)(x_0) := \lim_{\e\to0} \int_{\R^n \sm B_\e} (u(x_0+y)-u(x_0))K(y)dy.
\end{align*}
\end{remark}

\begin{proof}
For $\r \in (0,r]$, Consider the test function $u_\r = u_{(\varphi, B_\r(x_0))}$ touching $u$ from below. Notice that $\d(u_\r,x_0)$ increases to $\d(u,D\varphi(x_0),x_0)$ as $\r$ goes to zero. We have from the definition of viscosity solutions that,
\begin{align*}
 f(x_0) &\geq Iu_\r(x_0),\\
 &\geq \cM^-_{\cK}u_\r (x_0) - \b|D\varphi(x_0)|,\\
 &= \int \d^+(u_\r,x_0;y)K^-(y) - \d^-(u_\r,x_0;y)K^+(y)dy - \b|D\varphi(x_0)|.
\end{align*}
From $\d(u,D\varphi(x_0),x_0) \geq \d(u_\r,x_0)$ and the non negativity of $K^+$ we have
\begin{align*}
\d^-(u_\r,x_0)K^+ \geq \d^-(u,D\varphi(x_0),x_0)K^+ \geq 0,
\end{align*}
therefore $\d^-(u,D\varphi(x_0),x_0)K^+$ is integrable. Moreover, by monotone convergence
\begin{align*}
 \lim_{\r\to 0} \int \d^-(u_\r,x_0;y)K^+(y)dy = \int \d^-(u,D\varphi(x_0),x_0;y)K^+(y)dy.
\end{align*}
Now we use Fatou's Lemma, sending $\r$ to zero, on the inequality,
\begin{align*}
 f(x_0) + \b|D\varphi(x_0)| + \int \d^-(u_\r,x_0;y)K^+(y)dy &\geq \int \d^+(u_\r,x_0;y)K^-(y),
\end{align*}
to get
\begin{align*}
 &f(x_0) + \b|D\varphi(x_0)| + \int \d^-(u,D\varphi(x_0),x_0;y)K^+(y)dy,\\
 &=\lim_{\r\to0}\1f(x_0) + \b|D\varphi(x_0)| + \int \d^-(u_\r,x_0;y)K^+(y)dy\2,\\
 &\geq \liminf_{\r\to0} \int \d^+(u_\r,x_0;y)K^-(y)dy,\\
 &\geq \int \d^+(u,D\varphi(x_0),x_0;y)K^-(y)dy.
\end{align*}
(Notice that $\d^+(u_\r,x_0)$ doesn't necessarily converge to $\d^+(u,D\varphi(x_0),x_0)$ in $L^1(K^-)$ as $\r\to0$).

This shows that for every $K^- \leq K \leq K^+$ the function $\d(u,D\varphi(x_0),x_0)K$ is integrable, $L_K(u,D\varphi(x_0))(x_0)$ is well defined and also any inf-sup combination of $L_{K_{\a,\b}}(u,D\varphi)(x_0)$ with $K_{\a,\b}\in\cK$.

Moreover $(u-\varphi)(x_0+\cdot)K^+ = (\d(u,D\varphi(x_0),x_0)-\d(u_r,x_0))K^+$ is integrable and then for every $K \in\cK$ and $\r \in (0,r]$
\begin{align*}
&\int_{B_\r(x_0)} |(u-\varphi)(x_0+y)|K^+(y)dy,\\
&= \int_{B_\r(x_0)} |\d(u,D\varphi(x_0),x_0;y)-\d(u_\r,x_0;y)|K^+(y)dy,\\
&\geq |L_K(u,D\varphi(x_0))(x_0) - L_K u_\r(x_0)|,
\end{align*}
which, by absolute continuity, can be made arbitrarily small by taking $\r$ sufficiently small, independently of $K\in\cK$. Given $\e>0$ there exists some sufficiently small $\r>0$ such that,
\begin{align*}
&L_{K_{\a,\b}}(u,D\varphi(x_0))(x_0) + b_{\a,\b}\cdot D\varphi(x_0) - f(x_0),\\
&\leq L_{K_{\a,\b}}u_\r(x_0) + \e + b_{\a,\b}\cdot D\varphi(x_0) - f(x_0),\\
&\leq \e.
\end{align*}
Taking the supremum and the infimum and finally sending $\e$ to zero we conclude the result.
\end{proof}

This implies in particular that viscosity solutions that are in $C^{1,\s-1+\e}(\W)\cap L^1(\w_\s)$ (for some $\e>0$ arbitrarily small) are actually classical solution.

\begin{corollary}
Let $f$ be a continuos function and $I$ an operator of the form \eqref{inf_sub_operators} such that $\{L_{\a,\b}\}_{\a\in A, \b\in B} \ss \cL$. Given $\e\in(0,2-\s)$ and a viscosity super solution $u \in C^{1,\s-1+\e}(\W) \cap L^1(\w_\s)$ of $Iu \leq f$ in $\W$ then $u$ also satisfies the inequality classically.
\end{corollary}

\begin{proof}
By the proof of Lemma 4.2 in \cite{Caffarelli09}, $\tilde f = Iu$ is a continuous function so we need to show that $\tilde f \leq f$. By the previous Lemma $\tilde f \leq f$ at every point where $u$ can be touched from below. We will conclude by showing that such set of contact points is dense in $\W$.

Let $B_r(x_0) \ss \W$, by the modulus of continuity of $u$ there exists some $\d$ such that $u \geq u(x_0) + \d$ at $\p B_r(x_0)$. Consider then the following test function,
\begin{align*}
\varphi = u(x_0) - (\d/r)|x-x_0|^2 - \sup_{B_r(x_0)}(\varphi-u),
\end{align*}
in order to conclude that there exists a point in $B_r(x_0)$ where $u$ can be touched from below by a $C^2$ function.
\end{proof}


\subsubsection{Stability}\label{Subsubsec:Stability}

The following stability result can be proven as in Section 4 \cite{Caffarelli09}. See also Section 3 in \cite{Barles08} and Section 4 in \cite{Caffarelli11} for more general results where the operator $I$ is also approximated by a sequence of operators $\{I_k\}$ in a suitable sense.

\begin{lemma}[Stability]\label{lemma:Stability}
Let $\{f_k\}$ be sequence of continuos functions and $\{I_k\}$ a sequence of elliptic operators with respect to $\cL$. Let $u_k \in LSC(\W) \cap L^1(\w_\s)$ be a sequence of functions such that
\begin{enumerate}
\item $Iu_k \leq f_k$ in the viscosity sense in $\W$.
\item $u_k \to u$ in the $\G$ sense in $\W$.
\item $u_k \to u$ in $L^1(\w_\s)$.
\item $f_k \to f$ locally uniformly in $\W$.
\item $|u_k(x)|\leq C$ for every $x\in\W$.
\end{enumerate}
Then $Iu \leq f$ in the viscosity sense in $\W$.
\end{lemma}


\subsubsection{Comparison principle}\label{Subsubsec:Comparison_principle}

We state in this section three results. The first one is important as it allows to say that the same elliptic relation that holds in the classical sense, namely,
\begin{align*}
\cM_\cL^- (u-v) \leq Iu -Iv \leq \cM^+_\cL (u-v)
\end{align*}
also holds for viscosity solutions. The second result is the maximum principle for viscosity solutions of the extremal operators. As a consequence of both result we also obtain the comparison principle between sub and super solutions of the same equation. The proof of these results can also be obtained as in \cite{Caffarelli09} using int and sup convolutions of the solutions and the previously stated stability.

\begin{theorem}[Equation for the difference of solutions]\label{Eqdiff}
Let $I$ be a continuous elliptic operator with respect to $\cL$ and $f,g$ continuous functions. Given $u \in LSC(\W) \cap L^1(\w_\s)$ and $v \in USC(\W) \cap L^1(\w_\s)$ such that $Iu \leq f$ and $Iv \geq g$ hold in $\W$ in the viscosity sense, then $\cM^-_\cL(u - v) \leq f-g$ also holds in $\W$ in the viscosity sense.
\end{theorem}

\begin{theorem}[Maximum principle]
Let $u \in LSC(\bar\W) \cap L^1(\w_\s)$ be a viscosity super solution of $\cM^-_\cL u \leq 0$ in $\W$. Then $\inf_{\bar \W}u = \inf_{\R^n \sm \W}u$.
\end{theorem}

An important point to mention here with respect to our specific operators is that we have to check whenever the Assumption 5.1 in \cite{Caffarelli09} holds. This assumption requires the existence of a barrier used in the proof of the maximum principle stated above. Specifically it says that there is some constant $R_0 \geq 1$ large enough so that for every $R > R_0$, there exists a $\d > 0$ such that for any operator $L \in \cL$, we have that $L\varphi > \d$ in $B_R$, where $\varphi$ is given by
\begin{align*}
\varphi_R = \min(1,|x|^2/R^3),
\end{align*}
But our family is scale invariant therefore it is enough to check that the strict inequality can be attained for $\varphi_1$ in a neighborhood of the origin. Moreover, as our family of linear operators is continuous we only have to check that the strict inequality holds at the origin which is clearly true.

\begin{corollary}[Comparison principle]\label{coro:Comparison_principle}
Let $I$ be a continuous elliptic operator with respect to $\cL$, $u \in LSC(\bar\W) \cap L^1(\w_\s)$ be a viscosity super solution and $v \in USC(\bar\W) \cap L^1(\w_\s)$ be a viscosity sub solution of the same equation $Iw = f$ in $\W$. Then $u \geq v$ in $\R^n \sm \W$ implies $u \geq v$ in $\R^n$.
\end{corollary}


\subsubsection{Existence}\label{Subsubsec:Existence}

Existence and uniqueness of a solution in the viscosity sense follows from the comparison principle by using Perron's method, see \cite{Crandall92} for the local case and \cite{Barles08-2} for more related results in the non local case. The additional ingredient we need is the barrier from Lemma \ref{lemma:Barrier1} that guarantees that the boundary values are attained in a continuous way.

\begin{lemma}[Barrier]\label{lemma:Barrier1}
There exists some exponent $\a \in (0,1)$ and some radius $r_0 >0$ sufficiently small such that,
\begin{align*}
\varphi(x) = ((|x|-1)^+)^\a,
\end{align*}
satisfies,
\begin{align*}
\cM_\cL^+ \varphi < -1 \text{ in $B_{1+r_0}\sm\bar B_1$}.
\end{align*}
\end{lemma}

In the following proof $\a\to0$ as $\s\to2$. If we were looking for a boundary estimate which remains independent of $\s$ as it goes to two we would have to modify this proof a bit by using Lemma \ref{lemma:limit_to_the_classic}. However this barrier will only be used in the existence Theorem \ref{thm:existence} which is only a qualitative result for which we actually do not need the uniform estimate.

\begin{proof}
By the radial symmetry of $\varphi$ it is enough to show that the inequality holds for $x = (1+r)e_1$ with $r\in(0,r_0)$. It is also enough to show that it holds for $x = (1+r_0)e_1$. Indeed consider the rescaling $\tilde \varphi$ of $\varphi$ given by a dilation of magnitude $r_0/r$ and a translation such that $\tilde\varphi \leq \varphi$ and $\tilde\varphi((1+\r)e_1) = \varphi((1+\r)e_1)$ for $\r>0$, i.e.
\begin{align*}
\tilde\varphi(x) = (r_0/r)^\a\varphi((r/r_0)(x-e_1) + e_1).
\end{align*}
Then for every $L \in \cL$ we have there exists some rescaled $\tilde L \in \cL$ such that,
\begin{align*}
L\varphi ((1+r)e_1) &= (r_0/r)^{\s-\a}\tilde L\tilde\varphi((1+r_0)e_1),\\
&\leq \cM_\cL^+ \varphi((1+r_0)e_1).
\end{align*}
By taking the supremum over $L\in\cL$ on the left hand side we conclude that it is enough to prove that $\cM_\cL^+ \varphi((1+r_0)e_1) < -1$ to get the inequality in $B_{1+r_0}\sm\bar B_1$.

Now we just use the Stability Lemma \ref{lemma:Stability}. As $\a$ goes to 0, $\varphi$ converges to $\chi_{\R^n \sm B_1}$, locally uniformly in $\R^n \sm \bar B_1$, then $\cM_\cL^+ \varphi$ converges to $\cM_\cL^+ \chi_{\R^n \sm B_1}$ uniformly in $B_2 \sm B_{1+\e}$ for every $\e>0$. Notice that $\cM_\cL^+ \chi_{\R^n \sm B_1}(x)$ becomes arbitrarily negative as $|x| \to 1^+$. Take then $r_0$ such that $\cM_\cL^+ \chi_{\R^n \sm B_1}((1+r_0)e_1) < -2$ and $\a$ sufficiently small such that $\cM_\cL^+\varphi((1+r_0)e_1) < -1$.
\end{proof}

\begin{theorem}[Existence]\label{thm:existence}
Given a domain $\W \ss \R^n$ with the exterior ball condition, an continuous elliptic operator $I$ with respect to $\cL$ and $f$ and $g$ bounded and continuous functions (in fact $g$ only need to be assumed continuous at $\p\W$), then the Dirichlet problem, 
\begin{align*}
 Iu &= f \text{ in $\W$},\\
 u &= g \text{ in $\R^n\sm\W$},
\end{align*}
has a unique viscosity solution $u$.
\end{theorem}

\begin{proof}
The uniqueness part follows from the comparison principle. For the existence we have an stability result coming from Lemma \ref{lemma:Stability} and the comparison principle coming from Corollary \ref{coro:Comparison_principle} therefore Perron's method applies to show the existence of a viscosity solution $u$, defined as the smallest viscosity super solution above the boundary values given by $g$. Moreover we obtain $\|u\|_\8 \leq \|g\|_\8$. The Dirichlet boundary problem gets solved by $u$ provided that there exists barriers that force $u$ to take the boundary value in a continuous way. We will show that for $x_0 \in \p\W$ and $\e>0$, there exists some $\d>0$ such that we can find an continuous upper barrier $\psi$ that satisfies
\begin{align*}
\psi &\geq u \text{ in $\R^n$},\\
\psi &\leq g(x_0) + \e \text{ in $B_\d(x_0)$}.
\end{align*}
It would imply that $u$ is upper semicontinuous at $x_0$ and a similar argument with lower barriers shows that $u$ is continuous at $x_0$.

By translating the system of coordinates to $(x_0,g(x_0))$ we can assume $(x_0,g(x_0))=(0,0)$. Let $\d_0>0$ such that $|g(x)| < \e/2$ for $x \in B_{\d_0}$ and assume that $B_{\d_0/2}((\d_0/2)e_1) \ss B_{\d_0} \cap (\R^n \sm \W)$ by rotating the system of coordinates. Consider for $r \in (0,\d_0/2)$,
\begin{align*}
\psi_r(x) = \frac{\max(\|g\|_\8,\|f\|_\8)}{\varphi(r_0)}\max\1\varphi(r_0),\varphi\1\frac{x-re_1}{r}\2\2  + \e/2,
\end{align*}
where $\varphi$ and $r_0$ are the ones from Lemma \ref{lemma:Barrier1}.

\begin{figure}[h]
  \includegraphics[width = 7cm]{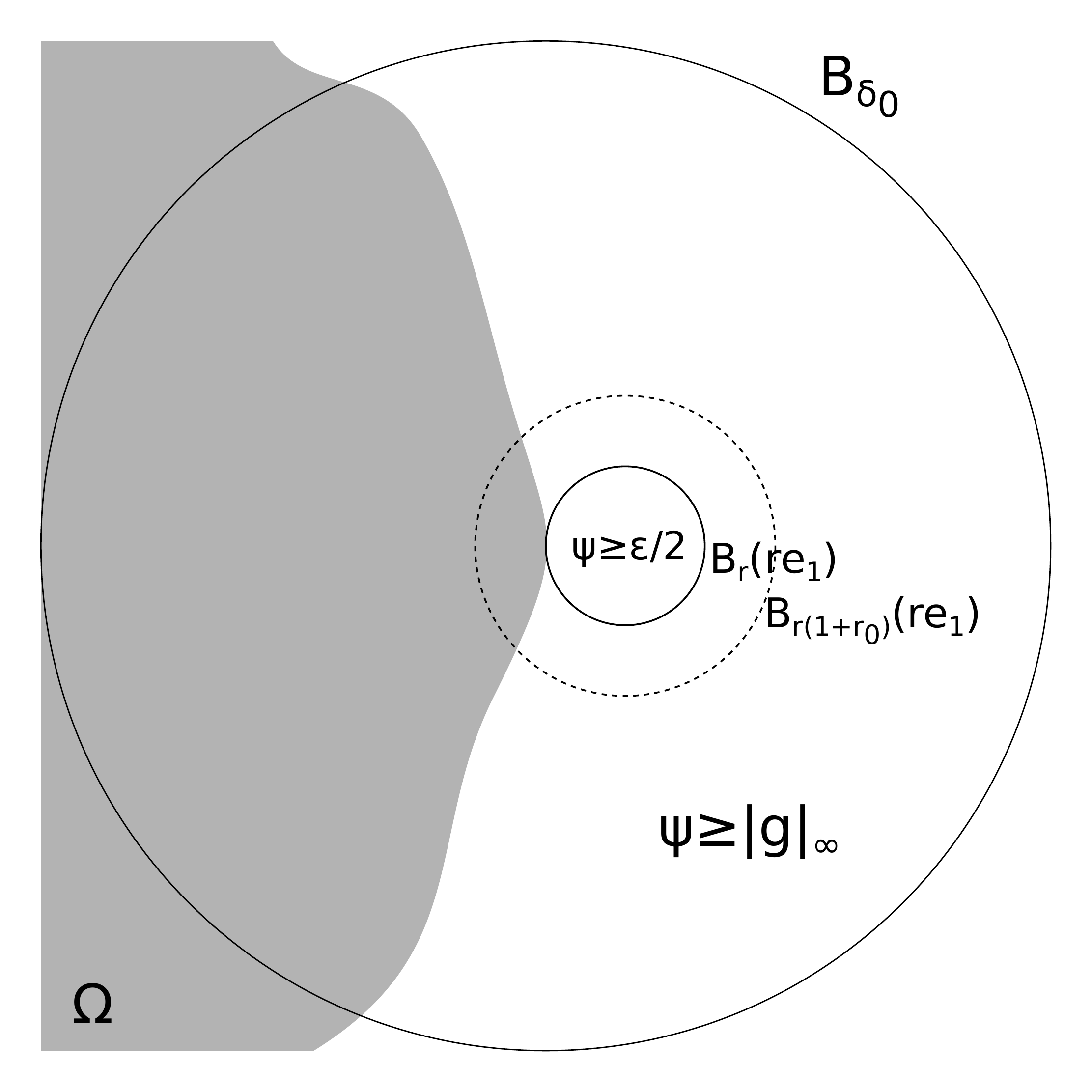}
  \caption{Barrier placement}
\end{figure}

By choosing $r = \d_0/(2+r_0)$ and $\psi = \psi_r = \psi_{\d_0/(2+r_0)}$, we obtain that $B_{r(1+r_0)}(re_1) \ss B_{\d_0/2}$ and then $\psi \geq g$ in $\R^n\sm\W$. Moreover $\psi \geq \|g\|_\8 \geq u$ in $\R^n \sm B_{r(1+r_0)}(re_1)$ which implies,
\begin{align*}
\psi \geq u \text{ in $\R^n \sm \1 B_{r(1+r_0)}(re_1) \cap \W\2$}.
\end{align*}

Because $\varphi$ is a supersolution outside $B_1$ we know that $\psi$ satisfies,
\begin{align*}
\cM^+_{\cL}\psi < -\|f\|_\8 \text{ in $B_{r(1+r_0)}(re_1) \sm \bar B_r(re_1) \supseteq B_{r(1+r_0)}(re_1) \cap \W$}.
\end{align*}
which by the comparison principle implies that $\psi \geq u$ in $\R^n$.

We use now the continuity of $\varphi$. There exists some $\d \in (0,r_0)$ such that
\begin{align*}
\varphi \leq \frac{\e\varphi(r_0)}{2\max(\|g\|_\8,\|f\|_\8)} \text{ in $B_{1+\d}$}.
\end{align*}
It implies that $\psi \leq \e$ in $B_\d$ with which we conclude the construction. 
\end{proof}

\section{Aleksandrov-Bakelman-Pucci estimate}\label{Sec:ABP_estimate}

In this section we prove an Aleksandrov-Bakelman-Pucci (ABP) estimate that start by assuming that the support of the positive part of the right hand side is well localized and concludes an estimate in measure which is also well localized. The main difficulty here is how to do that by having a gradient term with a bad sign. Lemma \ref{lemma:Barrier2} gives us a barrier that helps us to deal with it.

The first thing we do is to discuss the classical setup of the ABP estimate by introducing the convex envelope of the solution.

\begin{definition}[Convex envelope]\label{def:Convex_envelope}
For a function $u$ defined in $\R^n$, lower semicontinuous in $B_1$ and non negative in $\R^n \sm B_1$ we define the convex envelope $\G_u$ as the largest non positive convex function below $u$ in $B_{10}$, i.e.\ for $x \in B_1$
\begin{align*}
\G_u(x) &= \sup\{f(x): \text{$f$ is convex and $f \leq \min(u,0)$ in $B_{10}$}\},\\
&= \sup\{f(x): \text{$f$ is affine and $f \leq \min(u,0)$ in $B_{10}$}\}.
\end{align*}
For $x \in B_{10}\sm \R^n$ we just define $\G_u(x) = 0$. Whenever there is no chance of confusion we will use the notation $\G$ instead of $\G_u$.
\end{definition}

By the convexity of $\G$ we know that for any $x \in B_{10}$ there is always a plane that supports the graph of $\G$ at $(x,\G(x))$. By this we mean that for some $p \in \R^n$, $p\cdot(y-x) + \G(x) \leq \G(y)$ for every $y \in B_{10}$. We define in this way the (non empty) set of sub differentials of $\G$ at $x$, denoted by $D\G(x)$, as the set of slopes $p$ of the supporting planes of the graph of $\G$ at $(x,\G(x))$, i.e.\
\begin{align*}
D\G(x) = &\{p \in \R^n:p\cdot(y-x) + \G(x) \leq \G(y) \text{ for every $y \in B_{10}$}\}. 
\end{align*}
We denote also for $A \ss B_{10}$,
\begin{align*}
D\G(A) &= \bigcup_{x\in A}D\G(x),\\
\|D\G\|_{L^\8(A)} &= \sup\{|p|: p \in D\G(A)\}.
\end{align*}


\subsection{Hypothesis}

We can assume without lost of generality that $K^- = (2-\s)\l\chi_{B_4}|y|^{-(n+\s)}$. This follows by rescaling the equation by a factor of 4 which would spread the support of $K^-$. This assumption will be used in Lemma \ref{lemma:Barrier2}.

The hypothesis of this Section are the following ones. Let $f$ be a continuous function in $B_1$ bounded by above. Let $u$ satisfies,
\begin{alignat}{2}
\label{eq:ABP_hypothesis1} \cM^-_\cL u &\leq f &&\text{ in the viscosity sense in $B_1$},\\
\label{eq:ABP_hypothesis2} u &\geq 0 &&\text{ in $\R^n \sm B_1$},\\
\label{eq:ABP_hypothesis3} \sup_{B_1} u^- &\leq 1.
\end{alignat}
Eventually we will also require $\supp f^+$ to be contained in a ball $B_{\r_0}$ with $\r_0 \in (0,1)$ arbitrary for this part (it will be fixed in the proof of the point estimate).


\subsection{Main Lemma}

The main Lemma of this section and its first corollary are essentially the same as Lemma 8.1 in \cite{Caffarelli09}. We include its proof because of the technicalities we found when dealing with our difference operator $\d$, instead of the symmetric difference found in \cite{Caffarelli09}. We also choose to give a presentation where it can be noticed that the size of the configurations are not necessarily arbitrarily small. By this we mean that we can choose $k = \lfloor1/(2-\s)\rfloor$ in Corollary \ref{coro:Main_corollary}.

We fix now the following geometric configurations,
\begin{align*}
R_r(x) &= B_r(x) \sm B_{r/2}(x),\\
R_i(x) &= B_{r_i}(x) \sm B_{r_{i+1}}(x) \text{ for } r_i = \r_02^{-i}.
\end{align*}
For some $\r_0\in(0,1)$. We did not include the factor $2^{-1/(2-\s)}$ as in \cite{Caffarelli09} because we are not interested in recovering the classical ABP estimate as $\s$ goes to 2.

\begin{lemma}[Main Lemma]\label{lemma:Main_Lemma}
Given $\r_0\in(0,1)$, $u$ satisfying the hypothesis \eqref{eq:ABP_hypothesis1}, \eqref{eq:ABP_hypothesis2}, \eqref{eq:ABP_hypothesis3}, $P(x) = p\cdot x + h$ an affine function and $x_0 \in B_1$ and $M>0$ and $\m \in(0,1)$ such that:
\begin{enumerate}
\item $P \leq u$ in $B_{10}$,
\item $\sup_{B_1} P \geq -1$,
\item $M\m \r_0(1-2^{-(2-\s)k})\geq C$ for some universal $C>0$ to be fixed in the proof,
\item For every $i = 0,1,\ldots,(k-1)$,
\begin{align}\label{eq:Main_Lemma_hypotesis}
\frac{|\{u \geq P + (\|f^+\|_{L^\8(B_{r_k}(x_0))} + \b|p|)Mr_i^2\}\cap R_i(x_0)|}{|R_i(x_0)|} \geq \m.
\end{align}
\end{enumerate}
Then,
\begin{align}\label{eq:Main_Lemma_conclusion}
u \geq P + (\|f^+\|_{L^\8(B_{r_k}(x_0))} + \b|p|)r_k^2 \text{ in $B_{r_{k+1}}(x_0)$}.
\end{align}
\end{lemma}

\begin{proof}
Denote $C_0 = \|f^+\|_{L^\8(B_{r_k}(x_0))} + \b|p|$ and let $\varphi = \varphi_0((\cdot-x_0)/r_k)$ be a smooth bump function taking values between zero and one such that, $\supp \varphi_0 = B_{3/4}$ and $\varphi_0 = 1$ in $B_{1/2}$. We want to use the following function $v$ as a lower barrier to show that $u - P$ separates a distance $C_0r_k^2$ from zero in $B_{r_{k+1}}(x_0)$. Let
\begin{align*}
v = P\chi_{B_2} + 2C_0r_k^2(\varphi - 1/2).
\end{align*}

By contradicting \eqref{eq:Main_Lemma_conclusion}, we have that $u - v$ has a negative infimum at some point $x_1 \in B_{3r_k/4}(x_0)$. We use comparison principle given by Corollary \ref{coro:Comparison_principle} and Lemma \ref{lemma:Principal_value_lemma} to estimate $\cM^-(u-v)(x_1)$ from above. On the other hand we can estimate $\cM^-(u-v)(x_1)$ from below by using the hypothesis \eqref{eq:Main_Lemma_hypotesis}. This will allows us to fix $C$ sufficiently large on the third hypothesis to get a contradiction.

From Corollary \ref{coro:Comparison_principle} applied in $B_1$,
\begin{align*}
\cM^-_\cL(u-v) &\leq f - \cM^-_\cL v,\\
&\leq f + \b|p| + 2C_0r_k^{2-\s}\|\cM^-_\cK\varphi_0\|_\8 + 2C_0\b r_k\|D\varphi_0\|_\8 - \cM^-_\cK(P\chi_{B_2}),\\
&\leq C_0C - \cM^-_\cK(P\chi_{B_2}).
\end{align*}
We estimate the last term in $B_1$ by using the two following observations for $x \in B_1$. For $y \in B_1$, $\d(P\chi_{B_2},x;y) = \d(P,x;y) = 0$. For $y \in \R^n \sm B_1$ we use that the integral gets symmetrized, (recall that $K^- = 0$ outside of $B_4$),
\begin{align*}
\cM^-_\cK(P\chi_{B_2}) &\geq -\L(2-\s)\int_{\R^n\sm B_4} \frac{\d_e(P\chi_{B_2},x;y)^-}{|y|^{n+\s}},\\
2\d_e(P\chi_{B_2},x;y) &= P\chi_{B_2}(x+y) + P\chi_{B_2}(x-y) - 2P\chi_{B_2}(x).
\end{align*}
We show now that $\d_e(P\chi_{B_2},x;y)$ is non negative for $x\in B_1$. This is immediate if $(x\pm y) \in B_2$ where $P\chi_{B_2}$ is affine. In the other case we use that $P$ is non positive in $B_{10}$ and crosses the level set $-1$ in $B_1$ in order to get that $P\chi_{B_2} \geq -14/10$ and $P\chi_{B_2} \leq -10/12$ in $B_1$. It implies that, if $(x+y) \in \R^n\sm B_2$, then
\begin{align*}
2\d_e(P\chi_{B_2},x;y) = P\chi_{B_2}(x-y) - 2P\chi_{B_2}(x) \geq -\frac{14}{10}+2\frac{10}{12} > 0.
\end{align*}
We conclude from here that $\cM^-_\cL(u-v) \leq C_0C$ holds in $B_1$ in the viscosity sense.

At the minimum $x_1$, $(u-v)$ can be touched by below by a constant function (zero gradient). By using Lemma \ref{lemma:Principal_value_lemma} we have that $\cM^-_\cL(u-v,0)(x_1)$ is defined in the principal value sense and $\cM^-_\cL(u-v,0)(x_1) \leq CC_0$ (Go to Lemma \ref{lemma:Principal_value_lemma} and its following Remark for the definition of $\cM^-_\cL(\cdot,0)$). Now we estimate $\cM^-_\cL(u-v,0)(x_1)$ from the other side using that $(u-v)-(u-v)(x_1) \geq (u-P)\chi_{B_2} \geq 0$,
\begin{align*}
\cM^-_\cL(u-v,0)(x_1) &= (2-\s)\l\lim_{\e\to0}\int_{\R^n\sm B_\e}\frac{(u-v)(x_1+y) - (u-v)(x_1)}{|y|^{n+\s}}dy,\\
&\geq (2-\s)\l\sum_{i=1}^{k-1}\int_{x_1 + y \in R_i(x_0)}\frac{(u-P)(x_1+y)}{|y|^{n+\s}}dy.
\end{align*}
Given that $x_1 \in B_{3r_k/4}(x_0)$ we obtain that for $(x_1+y)\in R_i(x_0)$ with $i=0,1,\ldots,(k-1)$, $|y| \sim r_i$. Applying Chebyshev and using the hypothesis \eqref{eq:Main_Lemma_hypotesis},
\begin{align*}
\cM^-_\cL(u-v,0)(x_1) &\geq (2-\s)\l\sum_{i=1}^{k-1}C_0Mr_i^{2-\s}\m,\\
&\geq (2-\s)\l C_0M\m\frac{r_0^{2-\s}-r_k^{2-\s}}{1-2^{-(2-\s)}},\\
&\geq CC_0M\m \r_0(1-2^{-(2-\s)k}).
\end{align*}
By combining the inequalities we obtain that $C \geq M\m \r_0(1-2^{-(2-\s)k})$ from which we can get the contradiction with the third hypothesis of the Lemma.
\end{proof}


The following Corollary is just a contrapositive of the previous Lemma combined with Corollary 8.5 in \cite{Caffarelli09}. We omit the proof. 

\begin{corollary}\label{coro:Main_corollary}
Given $\r_0\in(0,1)$, $u$ satisfying the hypothesis \eqref{eq:ABP_hypothesis1}, \eqref{eq:ABP_hypothesis2}, \eqref{eq:ABP_hypothesis3} and $x_0 \in\{u=\G\}$ with $P(x) = p\cdot x + h$ its supporting plane. Then, for $k = \lfloor1/(2-\s)\rfloor$ and $M$ sufficiently large there is some radius $r \in \{r_0,r_1,\ldots,r_{k-1}\}$ such that:
\begin{enumerate}
\item
\begin{align*}
\frac{|\{u\leq P + (\|f^+\|_{L^\8(B_{r_k}(x_0))} + \b|p|)(M/\r_0)r^2\}\cap R_r(x_0)|}{|R_r(x_0)|} \geq 4^{-n}.
\end{align*}
\item
\begin{align*}
\G \leq P + (\|f^+\|_{L^\8(B_{r_k}(x_0))} + \b|p|)(M/\r_0)r^2 \text{ in $B_{r/2}(x_0)$}.
\end{align*}
\item
\begin{align*}
\frac{|D\G(B_{r/4}(x_0))|}{|B_{r/4}(x_0)|} \leq \1(\|f^+\|_{L^\8(B_{r_k}(x_0))} + \b|p|)(M/\r_0)\2^n.
\end{align*}
\end{enumerate}
\end{corollary}


The following barrier allow us to control how does $D\G$ spread from the support of $f^+$.

\begin{lemma}\label{lemma:Barrier2}
Given $u$ satisfying the hypothesis \eqref{eq:ABP_hypothesis1}, \eqref{eq:ABP_hypothesis2} and \eqref{eq:ABP_hypothesis3}, $\r_0 \in (0,1)$ and $\supp(f^+) \ss B_{\r_0}$ then there is some $\a>2$ sufficiently large and independent of $\s$ such that,
\begin{align*}
\varphi(x) = \1\1\frac{(|x|-\r_0)^+}{2}\2^\a-1\2\chi_{B_2}(x) \leq u \text{ in $B_1$}.
\end{align*}
\end{lemma}

\begin{proof}
We will show that $\cM^-_{\cL} \varphi > 0$ in $B_1 \sm \bar B_{\r_0}$, for $\a$ sufficiently large. Notice also that $\varphi \leq u$ in the complement of $B_1 \sm \bar B_{\r_0}$. The reason for this is because $\varphi \leq 0$ everywhere and $u$ can not take its minimum value $-1$ outside of the support of $f^+$. After proving that $\cM^-_{\cL} \varphi > 0$ in $B_1 \sm \bar B_{\r_0}$ we would get the conclusion $\varphi \leq u$ by applying the comparison principle in $B_1\sm \bar B_{\r_0}$.

Fix $x\in B_1 \sm \bar B_{\r_0}$. By using Lemma \ref{lemma:limit_to_the_classic} we have that, as $\s$ goes to two, $\cM^-_\cL\varphi$ goes to $\tilde\cM^- D^2\varphi - \b|D\varphi|$. By a standard computation for $x \in B_1\sm \bar B_{\r_0}$, $\a > 2$ and $|x| = r$,
\begin{align*}
\cM^- D^2\varphi(x) - \b|D\varphi(x)| &= \frac{\a(r - \r_0)^{\a-2}}{2^\a}\1\l\1(\a-1) + (n-1)\frac{r - \r_0}{r}\2\right.,\\
&\left. {} - \b(r-\r_0)\2,\\
&\geq \frac{\a(r - \r_0)^{\a-2}}{2^\a}\1 \l(\a-1) - \b\2.
\end{align*}
So for $\a-1 > \b/\l$ and $\s \in [2-\e,2)$, for some $\e$ sufficiently small, we get that $\cM^-_{\cL} \varphi > 0$ in $B_1 \sm \bar B_{\r_0}$.

Now we look at the case where $\s \in [1,2-\e)$. Fix $x\in B_1 \sm \bar B_{\r_0}$. In $B_2$, $\varphi$ is convex and then $\d(\varphi,x;y) \geq 0$ for $y \in B_1$. Therefore we just consider the integrals outside $B_1$ where it gets symmetrized 
\begin{align*}
\cM^-_\cK \varphi &\geq (2-\s)\int_{\R^n \sm B_1} \frac{\l\chi_{B_4}(y)\d_e(\varphi,x;y)^+ - \L\d_e(\varphi,x;y)^-}{|y|^{n+\s}}dy,\\
&\geq \e\int_{B_4 \sm B_1} \frac{\l\d_e(\varphi,x;y)^+ - \L\d_e(\varphi,x;y)^-}{|y|^{n+\s}}dy.
\end{align*}
For $x\in B_1$, the symmetric difference $\d_e(\varphi,x;y)$ is always non negative. If $(x\pm y) \in B_2$ then if follows by the convexity of $\varphi$ in $B_2$. In the other case we use that $\varphi \leq -1/2$ in $B_1$ (because $\a>1$) and $\varphi \geq -1$ in $\R^n$. If $(x+y) \in \R^n\sm B_2$ then,
\begin{align*}
2\d_e(\varphi,x;y) = \varphi(x-y) - 2\varphi(x) \geq -1 + 2\frac{1}{2} = 0.
\end{align*}
The previous estimate then can be continued in the following way,
\begin{align*}
\cM^-_\cK \varphi &\geq \e\l \int_{B_4 \sm B_1} \frac{\d_e(\varphi,x;y)^+}{|y|^{n+\s}}dy,\\
&\geq \frac{\e\l}{2} \int_{B_4\sm B_3} \frac{dy}{|y|^{n+\s}}dy,\\
&\geq \frac{\e\l|\p B_1|}{2}(3^{-\s} - 4^{-\s}).
\end{align*}
The last member of the inequality remains bounded from below by a constant $C>0$ uniformly for $\s \in [1,2-\e)$. Now we just have to choose $\a$ sufficiently large to make the remaining drift term sufficiently small in $B_1 \sm \bar B_{\r_0}$,
\begin{align*}
\cM^-_\cL \varphi &= \cM^-_\cK \varphi - \b|D\varphi|,\\
&\geq C - \b2^{-\a}\a(r-\r_0)^{\a-1},\\
&\geq C - \b2^{-\a}\a
\end{align*}
We conclude the Lemma in this case by chosing $\a$ such that $C/\b>\a2^{-\a}$.
\end{proof}


The following estimate results as a combination of Theorem 8.7 and the beginning of the proof of Lemma 10.1 in \cite{Caffarelli09} adapted to our case by using the barrier from Lemma \ref{lemma:Barrier2}

\begin{theorem}[ABP estimate]\label{thm:ABP_estimate}
Let $u$ satisfies the hypothesis \eqref{eq:ABP_hypothesis1}, \eqref{eq:ABP_hypothesis2} and \eqref{eq:ABP_hypothesis3}, $\r_0 \in (0,1)$ and $\supp(f^+) \ss B_{\r_0}$. Then, for some universal constants $c$ and $C$, independent of $\s$,
\begin{align*}
c\r_0^{n\a} \leq (\|f^+\|_\8^n + 1) |\{u \leq \G + C\r_0^{-1}(\|f^+\|_\8 + 1)\} \cap B|,
\end{align*}
where $B$ is the ball of radius $(2+16\sqrt{n})\r_0$ and the exponent $\a$ is the one from Lemma \ref{lemma:Barrier2}.
\end{theorem}

\begin{proof}
We use the barrier from Lemma \ref{lemma:Barrier2} in order to say that, for some universal $c$,
\begin{align}\label{eq:ABP}
D\G(B_{2\r_0}) \supseteq B_{c\r_0^{\a-1}}.
\end{align}
This can be justified by noticing that any plane with slope less than $\frac{(\r_0/2)^\a}{\r_0/3}$ can be brought from below such that it touches $\min(\varphi,\varphi(2\r_0))$ in $B_{\r_0}$ and is above $-1$ in $B_{\r_0}$. Therefore it crosses $u$ and can be translated down until it touches $u$ for the first time inside $B_{2\r_0}$.

\begin{figure}[h]
  \includegraphics[width = 9cm]{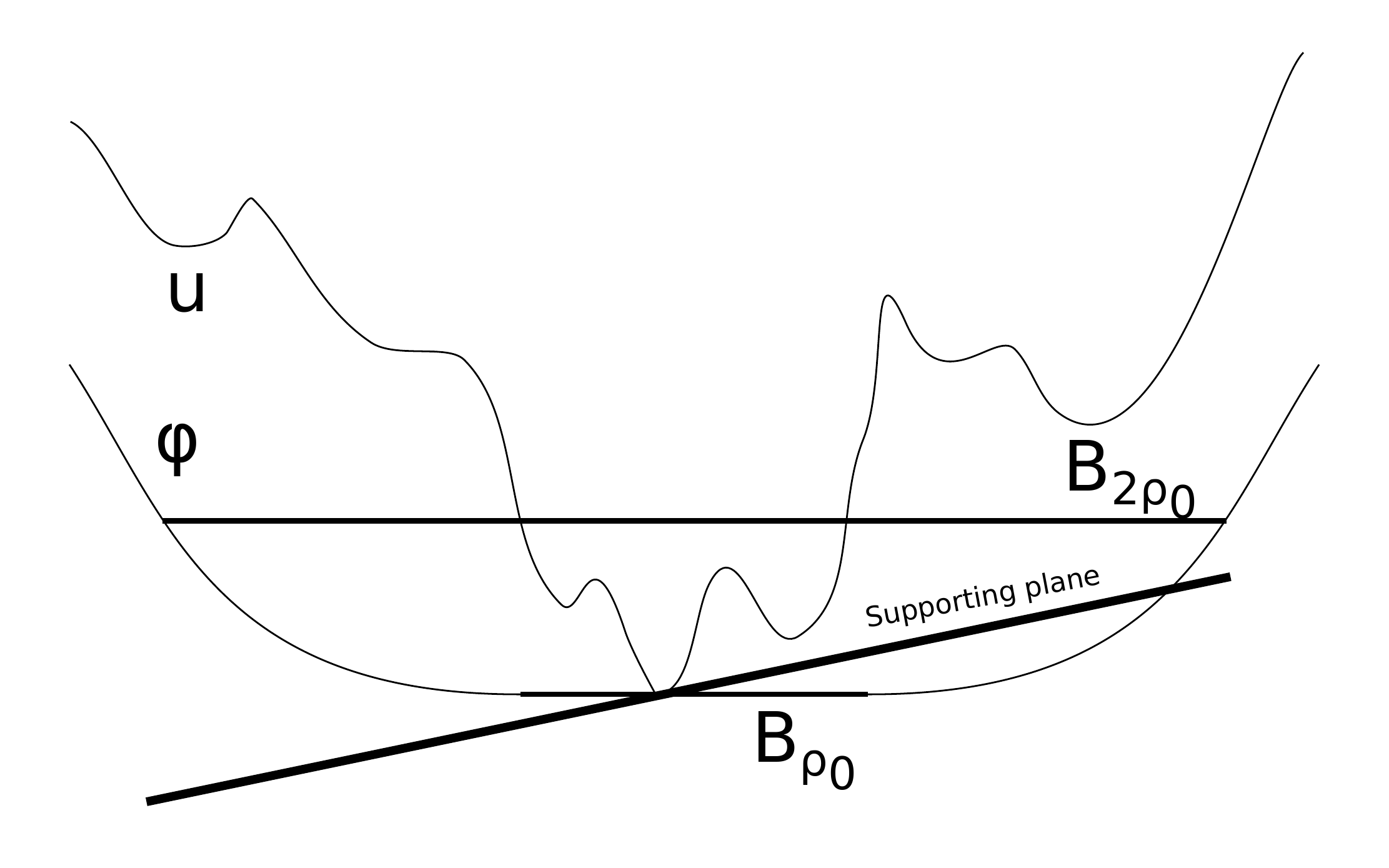}
\end{figure}

For the covering we start by finding a family $\{Q_i\}$ of $\{u=\G\} \cap B_{2\r_0}$ of dyadic cubes as in Theorem 10.1 in \cite{Caffarelli09} such that,
\begin{enumerate}
\item They are pairwise disjoints and every cube intersects $\{u=\G\} \cap B_{2\r_0}$.
\item They cover $\{u=\G\}\cap B_{2\r_0}$ and are all contained in $B_{3\r_0}$.
\item For some universal $C>0$,
\begin{align*}
\frac{|D\G(Q_i)|}{|Q_i|} \leq C\r_0^{-n}\1\|f^+\|_\8^n + 1\2.
\end{align*}
\item For some universal $\m \in (0,1)$, sufficiently small, and $C>0$, 
\begin{align*}
\frac{|\{u \leq \G + C\r_0^{-1}(\|f^+\|_\8 + 1)\}\cap 16\sqrt{n}Q_i|}{|16\sqrt{n}Q_i|} \geq \m
\end{align*}
\end{enumerate}
Notice that we have estimated $|p| \leq \|D\G\|_{L^\8(B_1)}$ by one because of the hypothesis \eqref{eq:ABP_hypothesis3}, combined with the fact that the support of $\G$ is sufficiently large.

By \eqref{eq:ABP} there exists a universal constant $c>0$ such that
\begin{align*}
c\r_0^{n(\a-1)} &\leq \sum |D\G(Q_i)|,\\
&\leq C\r_0^{-n}\1\|f^+\|_\8^n + 1\2 \sum |Q_i|,\\
&= C\r_0^{-n}\1\|f^+\|_\8^n + 1\2 \left|\bigcup Q_i\right|.
\end{align*}
Also $\{16\sqrt{n}Q_i\}$ covers $\bigcup Q_i$ and each one of this cubes is contained in $B$, the ball with radius $(2+16\sqrt{n})\r_0$. We extract then a sub covering $\{16\sqrt{n}\tilde Q_i\}$ with the finite intersection property and continue the estimate with,
\begin{align*}
\left|\bigcup Q_i\right| &\leq \sum |16\sqrt{n}\tilde Q_i|,\\
&\leq C\sum |\{u \leq \G + C\r_0^{-1}(\|f^+\|_\8 + 1)\}\cap 16\sqrt{n}\tilde Q_i|,\\
&\leq C\left|\{u \leq \G + C\r_0^{-1}(\|f^+\|_\8 + 1)\}\cap B\right|.
\end{align*}
Which concludes the proof.
\end{proof}

\section{H\"older regularity}\label{Sec:Holder_regularity}

By having an ABP estimate we can obtain as consequences a point estimate ($L^\e$ lemma, weak Harnack), a Harnack inequality and a H\"older modulus of continuity. We go directly to the H\"older regularity from the point estimate. We expect the Harnack inequality to hold similarly however we don't require it in our regularity result which follows \cite{Caffarelli09}.


\subsection{Point estimate}

A point estimate gives a way to control the distribution of positive super solutions by the infimum of the such solution. Next we fix some of the renormalized hypothesis.

\subsubsection{Hypothesis}

The hypothesis of this part need to be stated in domains sufficiently large such that we have enough room to construct a barrier which will localize the right hand side of the equation. We assume the following: Let $u$ satisfies,
\begin{alignat}{2}
\label{eq:PE_hypothesis1} \cM^-_\cL u &\leq 1 &&\text{ in the viscosity sense in $B_{4\sqrt{n}}$},\\
\label{eq:PE_hypothesis2} u &\geq 0 &&\text{ in $\R^n$},\\
\label{eq:PE_hypothesis3} \inf_{Q_3} u &\leq 1.
\end{alignat}
Because we need to apply a dilated version of the ABP estimate we will also be assuming that $K^- = (2-\s)\l\chi_{B_{16\sqrt{n}}}|y|^{-(n+\s)}$.


\subsubsection{A special function}

\begin{lemma}[Barrier]\label{lemma:Barrier3}
There exists some exponent $p>0$ sufficiently large and some radius $r_0 \in (0,2)$ sufficiently small, both independent of the order $\s \in [1,2)$, such that the function $\varphi = \min(|x|^{-p},r_0^{-p})$ satisfies $\cM^-_\cL \varphi > 0$ in $\R^n \sm B_2$.
\end{lemma}

\begin{proof}
By the rotational symmetry of $\varphi$ and the kernels $K^\pm$ we only need to show that the inequality gets satisfies for $x = re_1$ with $r\geq2$. It is also enough to prove it for $r=2$. For $r>2$ one would obtain it from the scaling of $\cL$ in the following way. Let $\varphi_r = (r-2)^{-p}\varphi(x/(r-2))$, then $\varphi_r \geq \varphi$, they coincide in $2e_1$ and from the previous considerations about the scaling of $\cL$,
\begin{align*}
\cM^-_\cL \varphi(re_1) &\geq C(r)\cM^-_\cL \varphi_r(2e_1) \geq C(r)\cM^-_\cL \varphi(2e_1),
\end{align*}
for some $C(r) >0$.

We prove the result now for $\s$ close to 2. From Lemma \ref{lemma:limit_to_the_classic},
\begin{align}\label{pucciTilde}
\lim_{\s\to2}\cM_\cK^- \varphi(2e_1) = \int_{\p B_1} \1\1\theta^t D^2 \varphi(2e_1)\theta\2^+ \l - \1\theta^t D^2\varphi(e_1)\theta\2^- \L \2d\theta,
\end{align}
where,
\begin{align*}
D^2 \varphi(e_1) = p2^{-(p+2)}\diag(p,-1,\ldots,-1).
\end{align*}
Therefore for every $p>p_0$, with $p_0$ sufficiently large depending on $\l,\L$ and $\b$, we can make the left hand side of \eqref{pucciTilde} greater that $p2^{-(p+2)}3\b$. It implies that there exists some $\e>0$ such that for every $\s \in (2-\e,2)$ we also have that $\cM_\cK^- \varphi(2e_1) > p2^{-(p+2)}2\b$. By taking $p_0$ even larger we get,
\begin{align*}
\cM_\cL^- \varphi(2e_1) = \cM_\cK^- \varphi(2e_1) - \b p2^{-(p+1)} > \b p2^{-(p+1)} > 0.
\end{align*}

For $\s \in[1,2-\e]$ we use the fact that for $p \geq n$, $|y|^{-p}$ is not integrable around the origin.
\begin{align*}
\cM_\cK^- \varphi(2e_1) &\geq \e \int \frac{\d^+(\varphi,2e_1;y)\l\chi_{B_1}(y) - \d^-(\varphi,2e_1;y)\L}{|y|^{n+\s}}dy,\\
&\geq \e \int_{B_1(-2e_1)} + \e\int_{\R^n \sm B_1(-2e_1)}.
\end{align*}
The first integral grows towards infinity as $r_0$ goes to zero if $p \geq n$. The last integral is always bounded by below because $\varphi \in C^2(B_1(2e_1)) \cap L^1(\w_1)$ uniformly as $p$ goes to infinity and $r_0$ goes to zero. Therefore for $r_0$ sufficiently small and $p$ even larger we can also get $\cM_\cK^- \varphi(2e_1) > p2^{-(p+2)}2\b$ and conclude as in the previous case.
\end{proof}


\subsubsection{Discrete point estimate}

\begin{lemma}[First estimate]\label{lemma:First_estimate}
Given $u$ satisfying the hypothesis \eqref{eq:PE_hypothesis1}, \eqref{eq:PE_hypothesis2}, \eqref{eq:PE_hypothesis3}  then for some universal constants $M>1$ and $\m \in(0,1)$,
\begin{align*}
\frac{|\{u > M\} \cap Q_1|}{|Q_1|} < \m 
\end{align*}
\end{lemma}

\begin{proof}
We construct from the function $\varphi$ in Lemma \ref{lemma:Barrier3} a non negative smooth barrier $\psi$, bounded by above, such that,
\begin{enumerate}
\item $\psi = 0$ in $\R^n \sm B_{4\sqrt{n}}$,
\item $\psi \geq 2$ in $B_{3\sqrt{n}} \supseteq Q_3$,
\item $\cM^-_\cL \psi > 2$ in $B_{4\sqrt{n}}\sm B_{1/(4+32\sqrt{n})}$.
\end{enumerate}
The following function (almost) does the work for some constants $C_0$ and $C_1$, 
\begin{align*}
\tilde\psi(x) &= C_0\varphi(C_1x).
\end{align*}
$C_1$ has to be chosen such that it concentrates the negative part of $\cM^-_\cL \tilde\psi$ inside $B_{1/(5+32\sqrt{n})}$. $C_0$ is chosen such that $\tilde\psi > 3$ in $B_{5\sqrt{n}}$ and $\cM^-_\cL \tilde\psi > 3$ in $B_{5\sqrt{n}}\sm B_{1/(4+32\sqrt{n})}$. Then we convolve $\tilde \psi$ with a smooth function in order to have a continuous right hand side. We also need to truncate it to make it zero outside $\R^n \sm B_{4\sqrt{n}}$. Clearly this can be made such that $\psi$ gives us all the conditions required.

Now we consider,
\begin{align*}
v(x) = \frac{(u - \psi)(4\sqrt{n}x)}{(4\sqrt{n})^\s}
\end{align*}
and verify the hypothesis of the ABP estimate with right hand side,
\begin{align*}
\cM_\cL^- v \leq 1 - \cM_\cL^-\psi(4\sqrt{n}\,\cdot) = \tilde f.
\end{align*}
Where $\tilde f^+$ is supported inside of $B_{\r_0}$ with $\r_0 = 1/(16\sqrt{n}+128n)$ and it is bounded by above by a universal constant. Then we have that for some universal constants,
\begin{align*}
(1-\m) \leq |\{v \leq \G_v + C\} \cap B_{(2+16\sqrt{n})\r_0}|.
\end{align*}
Which implies for $u$ that,
\begin{align*}
(1-\m) \leq |\{u \leq M\} \cap B_{1/2}| \leq |\{u \leq M\} \cap Q_1|.
\end{align*}
\end{proof}

By applying the previous Lemma at smaller scales and combining it with Lemma 4.2 in \cite{Caffarelli95} we get the following result. Here the fact that $\cL$ remains invariant by dilations is used in a fundamental way. 

\begin{corollary}[Discrete point estimate]\label{coro:Discrete_point_estimate}
Given $u$ satisfying the hypothesis \eqref{eq:PE_hypothesis1}, \eqref{eq:PE_hypothesis2}, \eqref{eq:PE_hypothesis3}  then for some $M>1$ and $\m \in(0,1)$ sufficiently small and every integer $k \geq 1$,
\begin{align*}
\frac{|\{u > M^k\} \cap Q_1|}{|Q_1|} \leq \m^k
\end{align*}
\end{corollary}

By rescaling the previous result and a standard covering argument we obtain a full point estimate.

\begin{theorem}[Point estimate]\label{thm:Point_estimate}
Let $u\geq 0$ satisfies,
\begin{align*}
\cM^-_\cL u &\leq f \text{ in the viscosity sense in $B_1$}.
\end{align*}
Then for some $\e,C>0$ universal constants and every $t>0$,
\begin{align*}
\frac{|\{u > t\} \cap B_{1/4}|}{|B_{1/4}|} \leq C\1\|f^+\|_\8 + \inf_{B_{1/2}} u\2^\e t^{-\e}
\end{align*}
\end{theorem}


\subsection{Regularity}

By having a Point estimate as above we also have H\"older regularity for solutions of equations which are elliptic with respect to $\cL$ by applying a diminish of oscillation argument. We refer to Lemma 12.2 in \cite{Caffarelli09} and Theorem 25 in \cite{Caffarelli11} for the proof. Again we notice that for the referred proofs to work in our setting it is fundamental that $\cL$ remains invariant by dilations.

\begin{theorem}[H\"older regularity]\label{thm:holder}
Consider $f$ a bounded continuous functions, $I$ and continuous elliptic operator with respect to $\cL$ (see Definition \ref{def:Ellipticity} and the hypothesis in Section \ref{Subsubsec:Hypothesis} imposed on $\cL$). Let $u$ be the viscosity solution of the Dirichlet problem,
\begin{align*}
Iu &= f \text{ in $B_1$},\\
u &= g \text{ in $\R^n \sm B_1$}.
\end{align*}
Then $u \in C^\a(B_{1/2})$ for some universal $\a\in(0,1)$ and satisfies,
\begin{align*}
\|u\|_{C^\a(B_{1/2})} \leq C(\|u\|_{L^\8(B_1)} + \|g\|_{L^1(\w_\s)} + \|f\|_{L^\8(B_1)} + \|I0\|_{L^\8(B_1)}),
\end{align*}
for some universal $C$.
\end{theorem}

Also for translation invariant equations we recover $C^{1,\a}$ regularity by considering incremental quotients of the solution. We also refer to Lemma 13.1 in \cite{Caffarelli09} and Theorem 27 in \cite{Caffarelli11} for the proof. In this case we need to add an additional assumption over the family of kernels. Let $\cK_1 \ss \cK$ defined by the additional restriction that for some fixed constant $C$, $|DK| \leq C|y|^{n+\s+1}$. Let $\cL_1 \ss \cL$ be the family of linear operators such that for every $L = L_K + b\cdot D$ we have that $K \in \cK_1$. 

\begin{theorem}[Regularity for translation invariant equations]\label{thm:holder_gradient}
Consider $f$ a bounded continuous functions, $I$ a translation invariant continuous elliptic operator with respect to $\cL_1$ (see Definition \ref{def:Ellipticity} and the hypothesis in Section \ref{Subsubsec:Hypothesis} imposed on $\cL \supseteq \cL_1$). Let $u$ be the viscosity solution of the Dirichlet problem,
\begin{align*}
Iu &= f \text{ in $B_1$},\\
u &= g \text{ in $\R^n \sm B_1$}.
\end{align*}
Then $u \in C^\a(B_{1/2})$ for some universal $\a\in(0,1)$ and satisfies,
\begin{align*}
\|u\|_{C^{1,\a}(B_{1/2})} \leq C(\|u\|_{L^\8(B_1)} + \|g\|_{L^1(\w_\s)} + \|f\|_{L^\8(B_1)} + \|I0\|_{L^\8(B_1)}),
\end{align*}
for some universal $C$.
\end{theorem}

{\bf Acknowledgment:} The author would like to thank Luis Caffarelli for proposing the problem and for various useful discussions.

\bibliographystyle{plain}
\bibliography{mybibliography}

\begin{thebibliography}{10}

\bibitem{Barles08-2}
G.~Barles, E.~Chasseigne, and C.~Imbert.
\newblock On the dirichlet problem for second-order elliptic
  integro-differential equations.
\newblock {\em Indiana Univ. Math. J.}, 57(1):213--246, 2008.

\bibitem{Barles08}
G.~Barles and C.~Imbert.
\newblock Second-order elliptic integro-differential equations: viscosity
  solutions' theory revisited.
\newblock {\em Ann. Inst. H. Poincar{\'e} Anal. Non Lin{\'e}aire},
  25(3):567--585, 2008.

\bibitem{Caffarelli09}
Luis Caffarelli and Luis Silvestre.
\newblock Regularity theory for fully nonlinear integro-differential equations.
\newblock {\em Comm. Pure Appl. Math.}, 62(5):597--638, 2009.

\bibitem{Caffarelli10}
Luis Caffarelli and Luis Silvestre.
\newblock On the evans-krylov theorem.
\newblock {\em Proc. Amer. Math. Soc.}, 138(1):263--265, 2010.

\bibitem{Caffarelli11}
Luis Caffarelli and Luis Silvestre.
\newblock Regularity results for nonlocal equations by approximation.
\newblock {\em Arch. Ration. Mech. Anal.}, 200(1):59--88, 2011.

\bibitem{Caffarelli95}
Luis~A. Caffarelli and Xavier Cabr{\'e}.
\newblock {\em Fully nonlinear elliptic equations}, volume~43.
\newblock American Mathematical Society Colloquium Publications, 1995.

\bibitem{Crandall92}
Michael~G. Crandall, Hitoshi Ishii, and Pierre-Louis Lions.
\newblock User's guide to viscosity solutions of second order partial
  differential equations.
\newblock {\em Bull. Amer. Math. Soc.}, 27(1):1--67, 1992.

\bibitem{Davila12}
G.~D{\'a}vila and H.~A.~Chang Lara.
\newblock Regularity for solutions of non local, non symmetric equations.
\newblock {\em Ann. Inst. H. Poincar{\'e} Anal. Non Lin{\'e}aire},
  http://dx.doi.org/10.1016/j.anihpc.2012.04.006, 2012.

\bibitem{Soner06}
Wendell~H. Fleming and Halil~Mete Soner.
\newblock {\em Controlled Markov Processes and Viscosity Solutions}, volume~25
  of {\em Stochastic Modelling and Applied Probability}.
\newblock Springer, New York, 2nd edition, 2006.

\bibitem{Krylov79}
N.~V. Krylov and M.~V. Safonov.
\newblock An estimate on the probability that a diffusion process hits a set of
  positive measure.
\newblock {\em Doklady Akademii Nauk SSSR}, 245(1):18--20, 1979.

\bibitem{Silvestre06}
Luis Silvestre.
\newblock H{\"o}lder estimates for solutions of integro-differential equations
  like the fractional laplace.
\newblock {\em Indiana Univ. Math. J.}, 55(3):1155--1174, 2006.

\end{thebibliography}

\end{document}